\documentclass[12pt]{amsart} %

\usepackage{amsmath}
\usepackage{mathrsfs}
\usepackage{amsfonts}
\usepackage{amssymb}
\usepackage{graphicx}
\usepackage{verbatim}
\usepackage{xcolor}
\usepackage{tikz}
\usepackage[margin=1in]{geometry}%
\usepackage[colorlinks,linkcolor=blue,citecolor=blue,pagebackref,hypertexnames=false, breaklinks]{hyperref}
\usepackage{cite}

\newtheorem{theorem}{Theorem}[section]

\newtheorem{corollary}[theorem]{Corollary}
\newtheorem{definition}[theorem]{Definition}

\newtheorem{lemma}[theorem]{Lemma}

\newtheorem{remark}[theorem]{Remark}
\newtheorem{conjecture}[theorem]{Conjecture}
\numberwithin{equation}{section}

\newcommand{\V}{\mathcal{V}}

\DeclareMathOperator{\nn}{\mathbb{N}}

\DeclareMathOperator{\real}{\mathbb{R}}
\DeclareMathOperator{\complex}{\mathbb{C}}
\DeclareMathOperator{\energy}{\mathscr{E}}
\DeclareMathOperator{\continuous}{\mathscr{C}}
\DeclareMathOperator{\domain}{\mathcal{D}}

\title
[Snowflake Domain with Boundary and Interior Energies]
{Discretization of the Koch Snowflake Domain with Boundary and Interior Energies}

\author[M. Gabbard]{Malcolm Gabbard}
\address[Malcolm Gabbard]{Colorado College\newline%
\indent Department of Mathematics and Computer Science}
\email{malcolm.gabbard@coloradocollege.edu}

\author[C. Lima]{Carlos Lima}
\address[Carlos Lima]{California State University Long Beach, Long Beach\newline%
\indent Department of Physics and Astronomy}
\email{carlos.lima@student.csulb.edu}

\author[G Mograby]{Gamal Mograby}
\address[Gamal Mograby]{University of Connecticut\newline%
\indent Department of Mathematics}
\email{gamal.mograby@uconn.edu}

\author[L.~G. Rogers]{Luke~G. Rogers}
\address[Luke Rogers]{University of Connecticut\newline%
\indent Department of Mathematics}
\email{luke.rogers@uconn.edu}

\author[A. Teplyaev]{Alexander Teplyaev}
\address[Alexander Teplyaev]{University of Connecticut\newline%
\indent Department of Mathematics}
\email{alexander.teplyaev@uconn.edu}

\begin{document}

\thanks{The authors are grateful to Kevin Marinelli for the support in implementing the numerical codes}
\thanks{Research supported in part by NSF DMS Grants 1659643  and 1613025.}
\date{\today}

\begin{abstract}
We study the discretization of a Dirichlet form on the Koch snowflake domain and its boundary with the property that both the interior and the boundary can support positive energy. We compute eigenvalues and eigenfunctions, and demonstrate the localization of high energy eigenfunctions on the boundary via a modification of an argument of Filoche and Mayboroda.  H\"older continuity and uniform approximation of eigenfunctions are also discussed. \tableofcontents
\end{abstract}

\maketitle
\pagestyle{headings}

\section{Introduction}

The main objective of this paper is to investigate a discrete version of the eigenvalue problem $\Delta u = \lambda u$ on the Koch snowflake domain, which we denote by $\Omega$. To this end,
we follow~\cite{MR3702726} and introduce a Dirichlet form (with a suitable domain) on  $\Omega$,
\begin{equation*}
\energy (u) := \int_{\Omega} (\nabla u)^2  d\mathcal{L}^2 + \energy_{\partial \Omega}(u|_{\partial \Omega}),
\end{equation*} 
where $\mathcal{L}^2$ is the usual Lebesgue measure on $\real^2$ and $\energy_{\partial \Omega}$
denotes the Kusuoka-Kigami Dirichlet form on the Koch snowflake boundary $\partial \Omega$. The novelty of this approach by comparison with past work~\cite{MR1337471,GrebenkovN13,MR994168,MR1412217,MR1350412,strichartz2019spectral} is that both the Euclidean interior and the fractal boundary carry non-trivial Dirichlet forms. We approximate the Dirichlet form $\energy$ by a sequence of discrete energies $\energy_n$ (Theorem~\ref{propositionDiscreteEnergy}). This is done by inductively constructing trianguations of $\overline\Omega=\Omega\cup\partial\Omega$, the edges of which are then treated as a sequence of finite planar graphs $\Gamma_n$ equipped with a discrete energy $\energy_n$ and measure $m_n$, and hence an inner product $\langle\cdot,\cdot\rangle$. We then define a discrete Laplacian $L_n$ such that
\begin{equation}
    \energy_n(u,v)=- \left\langle L_n u,v\right\rangle_n.
\end{equation}
Note that $L_n$ takes into account both the interior $\Omega$ and fractal boundary $\partial\Omega$. Our approach is related to several recent works on diffusion problems involving fractal membranes~\cite{MR1916407,MR3274752,MR3250804}. Moreover, it can be viewed as a generalization of work of Lapidus et.\ al.~\cite{MR1412217} on the eigenstructure of the Dirichlet Laplacian on $\Omega$.  Indeed, our numerical results on the spectra of $L_n$ are closely related to those for a discretization of the eigenvalue problem $\Delta u = \lambda u$ with Dirichlet boundary conditions due to a localization phenomenon that will be explored in Section~\ref{sec:numericalResults}.

One of the main goals of this paper is to investigate the impact of the fractal boundary on the eigenmodes of $L_n$. This problem is of general interest, particularly in physics where questions like ``How do ocean waves depend on the topography of the coastlines?'' and ``How do trees and wind interact?'' have already been studied. For these and more examples the reader is referred to \cite{PhysRevE.47.3013,sapoval1991vibrations} and references therein. We note in particular that Sapoval~\cite{SAPOVAL1989296} conducted an experimental investigation of acoustic vibration modes of with soap bubbles placed on fractal drums, and made the striking observation that the fractal boundary causes some low-frequency wave modes to localize.  Placing a soap bubble on the fractal boundary of a drum is mathematically equivalent to imposing Dirichlet boundary conditions; our situation is somewhat different, in that our model allows for non-trivial boundary energy, but nevertheless we find significant  localization effects, which may be summarized as follows:
\begin{itemize}
    \item The eigenvalue counting function has two regimes with different scaling. The threshold for the change in regimes is approximately the largest eigenvalue of the Dirichlet Laplacian on $\Omega$.
    \item Eigenfunctions with eigenvalues of $L_n$  below the threshold are not localized.
    \item Eigenfunctions with eigenvalues of $L_n$  above the threshold begin to show localization near $\partial\Omega$.  As the eigenvalue increases, so does this localization.
    \item Eigenfunctions corresponding to eigenvalues of $L_n$ that are significantly above the threshold are strongly localized on $\partial\Omega$.
\end{itemize}
A class of boundary-localized  high-frequency eigenmodes known as {\em whispering-gallery modes} are well-understood for convex domains, though an analytic explanation for their appearance in more general domains is not yet known~\cite{GrebenkovN13}.  However, our boundary modes do not seem to be whispering gallery modes, or any other known class of localized high-frequency modes. Rather, they appear to arise because the  boundary part of the Dirichlet form and the interior part have different scalings and hence interact only weakly.  An elementary way to see that such modes should appear in our model is given using a variant of the landscape map of Filoche and Mayboroda~\cite{filoche1,ADFJM} in Section~\ref{FilocheMayboroda Argument}.

The results given here are part of a long term study that aims to 
provide robust computational tools to address, in a fractal setting, a number of 
linear and nonlinear problems arising from physics
\cite{geim2000non,baelus2004vortex,akkermans1999vortices,akkermans2013statistical,dunne2012heat,akkermans2010thermodynamics}. 
The physics of magnetic fields, and of vector equations more generally, 
are particularly challenging on fractal spaces. 
Moreover, a discretization of the type considered here is expected to be essential in studying quantum walks~\cite{agliari2010quantum,ord1983fractal,Sabot12,aharonov1993quantum,kempe2003quantum}. On the abstract mathematical side, our work is   related to~\cite{Sabot06,Sabot07,hinz2018sobolev,capitanelli2014uniform,creo2018magnetostatic,brzoska2017spectra,basilica,HR,HT,j1,j2,j3,mosco1,mosco1994} and to the classical Venttsel's problem 
\cite{Apushkinskaya,Venttsel}. 

The paper is organized as follows. Section~\ref{sec:Koch Snowflake Domain}, follows the treatment in~\cite{MR3702726} to introduce a Dirichlet form on the Koch snowflake. In Section~\ref{sec:domainenergy} we give the Dirichlet form on the snowflake domain and discuss some of its properties, such as the H\"older continuity and uniform approximation of eigenfunctions. In Section~\ref{sec:Inductive mesh}
we construct a triangular grid to approximate the Koch snowflake domain and introduce the discrete Laplacian $L_n$.  Section~\ref{sec:numericalResults} is concerned with our algorithm and numerical results, including the existence of boundary-localized eigenfunctions and their effect on the eigenvalue counting function. Finally, in Section~\ref{FilocheMayboroda Argument} we show that the localized eigenfunctions $L_n$ can be predicted numerically using a variant of an argument from~\cite{filoche1}.


\section{Dirichlet form on the Koch snowflake}
\label{sec:Koch Snowflake Domain}
The Koch snowflake and the associated snowflake domain are well-known.  Here we introduce some notation and foundational results  for our analysis, following~\cite{MR3702726}.  Let 
$\left\{F_i \right\}_{i=1}^4$ be the iterated function system defined on $\complex$ by 
\begin{align*}
F_1 (z) &= \frac{z}{3}  &  
F_2 (z) &= \frac{z}{3} e^{i\frac{\pi}{3}} + \frac{1}{3}  \\
F_3 (z) &= \frac{z}{3} e^{-i\frac{\pi}{3}} + \frac{3+i\sqrt{3}}{6}& 
F_4 (z) &= \frac{z+2}{3}. 
\end{align*}
The Koch curve is the unique nonempty compact subset 
$K$ of $\complex$ such that $K=\bigcup_{i=1}^4 F_i (K)$.  It can be approximated by a sequence of finite graphs for which the following notation is convenient.
\begin{definition} 
Let $S = \{ 1,2,3,4 \}$. We define $W = S^{\nn} $ and call $\omega \in W$ an infinite word. Similarly, a finite word of length $ n \in \nn$ is $w\in S^n$; we write $|w|=n$ for its length.
\end{definition} 
We write $F_{w} := F_{w_n}\circ \ldots \circ F_{w_1}$, where $w = w_1 \ldots w_n \in S^n$ and introduce finite graphs approximating the Koch curve as follows.
\begin{definition}\label{def:Vn}
Let $V_0(K)=\left\{0,1\right\} \subset \complex $ and $V_{w}(K):=F_{w}(V_0(K))$ for a finite word $w$. Then define
\begin{equation*}
V_n(K):=\bigcup_{|w|  =n} V_{w}(K),  \text{ and }\,  V_{\ast}(K):=\bigcup_{n \geq 0} V_n(K). 
\end{equation*}
We consider the points of $V_n(K)$ as vertices of a graph in which adjacency, denoted by $p\sim_n q$, means that there is a word $w$ of length $n$ such that $p,q \in V_w(K)$.
\end{definition}

On each of these graphs we  define a graph energy for  $u : V_{\ast}(K) \to \real$ by
\begin{equation}\label{eqn:bdygraphenergy}
\energy_K^{(n)}(u) = \frac{4^n}{2} \sum_{p \in V_n(K)} \sum_{q \underset{n}{\sim} p} \left(u(q)-u(p)\right)^2. 
\end{equation}
Following the general treatment in~\cite{MR1840042}, to which we refer for all omitted details, we see that  $\{\energy_K^{(n)}(u) \}$ is  nondecreasing, so $\energy(u) = \lim_{m \to \infty} \energy_m(u)$ is well-defined; setting its domain to be $\{ u:V_{\ast}(K) \to \real \ | \ \energy_K(u)< \infty \}$ one obtains a resistance form.  This non-negative definite, symmetric quadratic form extends to
$\domain{(\energy_K)} := \{u \in  \continuous(K)  \ | \ \energy_K(u|_{V_{\ast}})< \infty \}$,
where $\continuous(K)$ is the space of continuous functions on $K$.   There is a resistance metric $R(x,y)$ on $K$ defined from $\energy_K$ and with the property that points $x,y\in V_n$ with $x\sim_n y$ have $R(x,y)$ comparable to $4^{-n}$, and thus $R(x,y)$ is bi-H\"older to the Euclidean metric on $K$ with $R(x,y)\asymp|x-y|^{\frac{\log4}{\log3}}$.  There is a resistance estimate for  $f\in\domain{(\energy_K)}$
\begin{equation}\label{eq:resistbd}
	|f(x)-f(y)|^2\leq R(x,y)\energy_K(f)
\end{equation}
so in particular these functions are $\frac{\log2}{\log 3}$-H\"older in the Euclidean metric, see also~\cite[Corollary~4.8]{MR2066098}.

Continuing our use of results from~\cite{MR1840042}, we see that if equip the Koch curve with the standard Bernoulli probability measure $\mu_K$, i.e. the self-similar measure with
weights $\{\mu_i \}^{4}_{i=1}$ and $\mu_i=\frac{1}{4}$ for $i \in \left\{1,2,3,4\right\} $ then $(\domain{(\energy_K)},\energy_K)$ gives a strongly local regular Dirichlet form on $L^2(K,\mu_K)$.

A particular collection of functions in $\domain (\energy_K)$ of $\energy_K$ will be useful in what follows.  A function $h$ on $K$ is called harmonic if it minimizes the graph energies $\energy_K^{(n)}(h)$ for all $n \geq 1$.  It is called piecewise harmonic at scale $n$ if it is harmonic on the complement of $V_n$, or equivalently if $h\circ F_w^{-1}$ is harmonic for each word $w$ of length $n$.  Piecewise harmonic functions are uniform-norm dense in $\continuous(K)$ and dense in $\domain(\energy_K)$ with respect to the norm $\bigl( \|u\|_{L^2(K,\mu_K)}^2+\energy_{K}(u)\bigr)^{1/2}$.

As in~\cite{MR2066098}, we transfer the above definitions for the Koch curve to the boundary of the snowflake domain $\overline{\Omega}$ in a the obvious manner.  Write the boundary $\partial \Omega$  as a union $\cup_j K_j$ of three congruent copies of $K$.  Specifically, let $K_j=\varphi_j(K)$ where $\varphi_j$ is the Euclidean translation and rotation such that $0\mapsto \sqrt{3}e^{i(4j-3)\pi/6}$ and $1\mapsto \sqrt{3}e^{i(4j+1)\pi/6}$.  
\begin{definition}
The boundary energy and its domain are defined by:
\begin{align*}
\domain{(\energy_{\partial \Omega})} &:= \left\{ u:\partial \Omega \to \real \ | \ u|_{K_i}\circ \varphi_i^{-1} \in \domain(\energy_{K}), i=1,\ldots,3  \right\},\\
\energy_{\partial \Omega}(u) &:= \energy_{K}(u|_{K_1}\circ \varphi_1^{-1})+\energy_{K}(u|_{K_2}\circ \varphi_2^{-1})+\energy_{K}(u|_{K_3}\circ \varphi_3^{-1}) \text{ if }u \in \domain(\energy_{\partial \Omega}).
\end{align*}
\end{definition}
We then let
\begin{equation}
\label{eq:measureBoundary}
\mu(\cdot):=\mu_{K}(\varphi_1^{-1}( \cdot))+\mu_{K}(\varphi_2^{-1}( \cdot))+\mu_{K}(\varphi_3^{-1}( \cdot))
\end{equation}
so that  $(\energy_{\partial \Omega},\domain{(\energy_{\partial \Omega})})$ is a strongly
local Dirichlet form on $L^2(\partial \Omega,\mu)$. Since $\mu_K(K)=1$ we have $\mu(\partial \Omega)=3$.

\section{Dirichlet form on the snowflake domain}\label{sec:domainenergy}

We wish to consider a Dirichlet form on $\overline\Omega$ that incorporates our form $\energy_{\partial\Omega}$ as well as the classical Dirichlet energy on $\Omega$. The latter is, of course, simply $\int_\Omega |\nabla f|^2\,d\mathcal{L}^2$, where $ \mathcal{L}^2$ is Lebesgue measure. The domain is the Sobolev space $H^1$.  The fact that a nice Dirichlet form of this type exists depends on results of Wallin~\cite{Wallin} and Lancia~\cite{MR1916407}, which we briefly summarize.  

The trace of $H^1(\Omega)$ to $\partial\Omega$ is well defined and can be identified with a Besov space $B$, the details of which will not be needed here~\cite{Wallin}. Moreover, the kernel of the trace map is $H^1_0(\Omega)$, the $H^1$-closure of $C^\infty$ functions with compact support in $\Omega$.   The domain $\domain{(\energy_{\partial \Omega})}$ can be identified with a closed subspace of $B$, and there is a bounded linear extension operator from $\domain{(\energy_{\partial \Omega})}$ to $H^1(\Omega)$, see~\cite{MR1916407}.  Writing $u|_{\partial\Omega}$ for the trace of $u$ we see that the following defines a Hilbert space and inner product, see~\cite[Proposition 3.2]{MR3250804}.
\begin{align*}
	W(\Omega,\partial\Omega) &= \{u\in H^1(\Omega): u|_{\partial\Omega}\in \domain{(\energy_{\partial \Omega})}\}\\
\langle u,v \rangle_{V( \Omega, \partial \Omega)} &= \langle u,v \rangle_{H^1(\Omega)}+\energy_{\partial \Omega}(u|_{\partial \Omega}, v|_{\partial \Omega})+\langle u|_{\partial \Omega}, v|_{\partial \Omega}\rangle_{L^2(\partial \Omega,\mu)}. 
	\end{align*}
	
One consequence of the preceding is that if we let $m= \mathcal{L}^2 |_{\Omega} + \mu|_{\partial \Omega}$ where $\mathcal{L}^2$ is Lebesgue measure on $\Omega$ and $\mu$ is from \eqref{eq:measureBoundary} then for any $c_0 > 0$ the quadratic form 
\begin{equation}\label{e-key}
\energy (u) := \int_{\Omega} (\nabla u)^2  d\mathcal{L}^2 +c_0 \energy_{\partial \Omega}(u|_{\partial \Omega}).
\end{equation}
with domain $W( \Omega, \partial \Omega))$ in $L^2( \overline{\Omega},m)$ is a Dirichlet form.

Another consequence that will be significant in the next section is that $W( \Omega, \partial \Omega))$ may be written as the sum of $H^1_0(\Omega)$ and the subspace of $H^1(\Omega)$ obtained by the extension operator from $\domain{(\energy_{\partial \Omega})}$.  In this context it is useful to describe an explicit extension operator given in~\cite{MR3702726} as the ``second proof'' of their Theorem~6.1, and to extract some further features of the extension and their consequences. We refer to~\cite{MR3702726} for a more detailed exposition of the construction, including diagrams of the hexagons and the triangulation.

The extension operator for a function $f$ on $\partial\Omega$ is defined as follows.  
There is an induction over $n=0,1,2,\dotsc$ that defines an exhaustion of $\Omega$ by regular hexagons of sidelength $3^{-n}$ which meet only on edges.  As the induction proceeds, each new hexagon is subdivided into $6$ equilateral triangles meeting at the center and a piecewise linear function is defined on each triangle by linear interpolation of the values at the vertices.  The value at the center of the hexagon is defined to be the average of its vertices, and the induction is such that the hexagon vertices either lie on edges of the triangles from the previous stage of the construction, which provides the values at these points, or lie on $\partial\Omega$. In the latter case the values of $f$ are used at these vertices, and we note that those vertices of a hexagon with sidelength $3^{-n}$ that come from $\partial\Omega$ are points from the vertex set $V_{n+1}$, see Definition~\ref{def:Vn}.  The resulting piecewise linear function is the extension of $f$ to $\Omega$, which we call $g$. It is a special case of the result in~\cite[Theorem~6.1]{MR3702726} that if $f$ is a piecewise harmonic function then $g$ is Lipschitz in the Euclidean metric on $\overline\Omega$.

With regard to the following theorem, we note that the existence of a bounded linear extension in this setting is known~\cite{MR1916407}, and we could develop the corollaries from this, but it will be useful for us to know the result for this specific extension.

\begin{theorem}\label{thm:ext}
The extension operator described above is a bounded linear extension operator from $\domain{(\energy_{\partial \Omega})}$ with norm $(\|\cdot\|_{L^2(\partial\Omega,\mu)}^2+\energy_{\partial\Omega}(f))^{1/2}$ to $H^1(\Omega)$ with its usual norm, and satisfies the seminorm bound
\begin{equation}
\label{eq:seminormbdforextension}
\int_\Omega |\nabla g|^2\leq C'\sum_n\sum_{x\underset{n}{\sim}y} |f(x)-f(y)|^2\leq C\energy_{\partial\Omega} (f).
\end{equation}
\end{theorem}

We need the following elementary lemma.
\begin{lemma}\label{lem:energyisdiffssquared}
If $h$ is a linear function on an equilateral triangle $T\subset\mathbb{R}^2$ then $\int_T |\nabla h|^2$ over the triangle is a constant multiple of the sum of the squared differences between vertices, independent of the size of the triangle. Thus $\int |\nabla g|^2$ over a hexagon in the extension construction is bounded by a constant multiple of the squared vertex differences summed over all pairs of vertices of the hexagon.
\end{lemma}

\begin{proof}
The first statement follows from a direct computation for a triangle of sidelength $1$ that is left to the reader, combined with the observation that when rescaling length the scaling of $|\nabla h|^2$ cancels with that for the area of the triangle. For the second, write the integral as a sum over the triangles, use the first statement on each triangle and apply the Cauchy-Schwarz inequality.
\end{proof}

\begin{proof}[\protect{Proof of Theorem~\ref{thm:ext}}]
We begin by observing that all operations in the definition of the extension are linear in $f$, and consequently so is the extension operator. In the following argument $C$ is a constant that may change value from step to step, even within an inequality.

Observe that in the construction, the vertex values of a hexagon of sidelength $3^{-n}$ come either from values at $f$ at points of $V_n$, or from linear interpolation of values of $f$ at vertices from $V_{n-1}$ on the side of a hexagon of the previous scale, or from linear interpolation between a value at a point in $V_{n-1}$ and a value obtained as an interpolant between values at vertices from $V_{n-2}$. The observation that drives the  Lipschitz bound in~\cite{MR3702726} is that in any of these cases we can bound the  pairwise difference between values at vertices by $C\sum |f(x)-f(y)|$ where the sum ranges over neighbor pairs in $V_n$ that are within a bounded distance from the hexagon.  It follows that the number of terms in the sum has  a uniform bound, and thus the squared pairwise differences are bounded by $C\sum |f(x)-f(y)|^2$.
In view of Lemma~\ref{lem:energyisdiffssquared}, the left side of~\eqref{eq:seminormbdforextension} is bounded by a constant multiple of the sum of squares over edges in the triangulation.  The preceding argument gives a bound for the sum over the edges that are in a hexagon of size $3^{-n}$, and summing over the hexagons yields~\eqref{eq:seminormbdforextension}.

For the norm bound we must control $\|g\|_{L^2(\Omega)}$.  It is immediate from the construction that $\|g\|_{L^\infty(\Omega)}\leq\|f\|_{L^\infty(\partial\Omega)}$. However the latter can be bounded in a standard manner from the resistance bound~\eqref{eq:resistbd}, because it implies that $|f(x)|\geq\frac12\|f\|_\infty$ on an interval of size controlled by $\energy_{\partial \Omega}$. Direct computation then gives $\|f\|_\infty^2\leq 2(\|f\|_{L^2(\partial\Omega)}^2+\energy_{\partial\Omega}(f))$. Since $\|g\|_{L^2(\Omega)}\leq C\|g\|_{L^\infty(\Omega)}$ we obtain $\|g\|_{L^2(\Omega)}\leq C \|f\|_{L^2(\partial\Omega)}$.  Together with the seminorm bound~\eqref{eq:seminormbdforextension} this proves the extension operator is bounded.
\end{proof}

\begin{corollary}\label{cor:domainofform}
The domain of the Dirichlet form $\energy (u) := \int_{\Omega} (\nabla u)^2  d\mathcal{L}^2 +c_0 \energy_{\partial \Omega}(u|_{\partial \Omega})$ may be written as a sum of compact subspaces of $L^2(\Omega)$
\begin{equation*}
W( \Omega, \partial \Omega) = H^1_0(\Omega)+ H^1_{\energy(\partial\Omega)}(\Omega)
\end{equation*}
where $H^1_{\energy(\partial\Omega)}(\Omega)$ is the image of $\domain(\energy_{\partial\Omega})$ under the extension map.
\end{corollary}
\begin{proof}
The decomposition as a sum of closed subspaces follows from Wallin's result~\cite{Wallin} that the kernel of the trace map is $H^1_0(\Omega)$, and from Lancia's bounded linear extension~\cite{MR1916407} (or from Theorem~\ref{thm:ext} above).  Compactness of $H^1_0(\Omega)$ is from the classical Rellich–Kondrachov theorem.  Compactness of $H^1_{\energy(\partial\Omega)}(\Omega)$ is an immediate consequence of the density of the finite dimensional space of harmonic functions in  $\domain(\energy_{\partial\Omega})$  and the boundedness of the extension map in Theorem~\ref{thm:ext} because it exhibits $H^1_{\energy(\partial\Omega)}(\Omega)$ as the completion of a sequence of finite dimensional spaces in $H^1(\Omega)$.
\end{proof}

The following consequence is standard.
\begin{corollary}\label{cor:efnsexist}
The non-negative self-adjoint Laplacian $L$ associated to the Dirichlet form by $\energy(u,v)=-\int_{\overline{\Omega}} (Lu)v\,dm $ has compact resolvent and thus its spectrum is a sequence of non-negative eigenvalues accumulating only at $\infty$.
\end{corollary}

We also note that the extension construction gives us an explicit H\"older estimate.
\begin{corollary}
Functions in $H^1_{\energy(\partial\Omega)}$ are $\frac{\log2}{\log3}$-H\"older in the Euclidean metric on $\overline\Omega$.
\end{corollary}
\begin{proof}
The resistance estimate~\eqref{eq:resistbd} says that a function $f\in\domain(\energy_{\partial\Omega})$ is $\frac12$-H\"older in the resistance metric.  A pair of neighboring points in $x,y\in V_n$ are separated by resistance distance $R(x,y)\sim4^{-n}$, so $|f(x)-f(y)|\leq C2^{-n}$ for such points. Moreover these points are separated by Euclidean distance $3^{-n}$, so $f$ is $\frac{\log2}{\log3}$-H\"older in the Euclidean metric.

The values of $f$ on $V_n$ are those used as boundary data on hexagons of sidelength $3^{-n}$ in the extension construction, and they obviously contribute a term with gradient bounded by $2^n3^{-n}$.  Since the extension $g$ on such a hexagon involves terms from the extensions to hexagons of scale $3^{-m}$ for $m\leq n$, we may sum these to see that $|\nabla g|\leq C2^n3^{-n}$ on this hexagon.  In particular, $|g(x)-g(y)|\leq C|x-y|^{\log2/\log3}$.  This is true for hexagons of all scales, hence for all points in $\Omega$, and was already known for points on $\partial\Omega$, so the proof is complete.
\end{proof}

There are discontinuous functions  in $W(\Omega,\partial\Omega)$ because it contains $H^1_0(\Omega)$.  However, eigenfunctions of $\energy$, which exist by Corollary~\ref{cor:efnsexist}, can be shown to be H\"older continuous.

\begin{theorem}\label{thm:efnsHolder}
Eigenfunctions of the Dirichlet form \eqref{e-key} (equivalently of $L$ as in Corollary~\ref{cor:efnsexist}) are H\"older continuous. 
\end{theorem}
\begin{proof}[Sketch of the proof]
Suppose $u$ is an eigenfunction with eigenvalue $\lambda$, so $\energy(u,v)=-\lambda\langle u,v\rangle_{L^2(dm)}$ for all $v\in\domain(\energy)$, so by Corollary~\ref{cor:domainofform} it is in particular true for all $v\in H^1_0(\Omega)$. We will write $\energy_\Omega(u,v)=\int_\Omega  \nabla u \cdot \nabla v\, d\mathcal{L}^2$ for the usual seminorm in $H^1(\Omega)$.

Let $h_u$ be the harmonic function on $\Omega$ with boundary data $u|_{\partial\Omega}$, where we recall that the latter denotes the trace.  Since $h_u$ is the extension of $u|_{\partial\Omega}$ that minimizes $\energy_\Omega$ and we know $u$ is an extension for which $\energy_\Omega(u,u)<\infty$, we have $h_u\in H^1(\Omega)$ and it follows immediately that $h_u\in\domain(\energy)$, and in fact $u-h_u\in H^1_0(\Omega)$.  

Then for $v\in H^1_0(\Omega)$ we compute 
\begin{align*}
\energy_\Omega(u-h_u,v)
= \energy_\Omega(u,v)
= \energy(u,v)
= -\lambda\langle u,v\rangle_{L^2(dm)}
= -\lambda\langle u,v \rangle_{L^2(\Omega,d\mathcal{L}^2)}.
\end{align*}
This uses the fact that 
$\energy_\Omega(h_u,v)=0$ because $h_u$ is harmonic and that   boundary terms vanish because $u-h_u=0$ on $\partial\Omega$ and $v|_{\partial\Omega}=0$, the latter because $H^1_0(\Omega)$ is the kernel of the trace map.

This shows that $u-h_u\in H^1_0(\Omega)$ satisfies $\Delta_0(u-h_u)=\lambda u$, where $\Delta_0$ is the Dirichlet Laplacian. This reduces the problem to determining the regularity of the harmonic function $h_u$ and $u-h_u=\lambda G_0u$, where $G_0$ is the Dirichlet Green's operator on $\Omega$.

It is proved in~\cite{n2} that the kernel $g(x,y)$ of the Green's operator is in a Sobolev space $W^{1,q}(\Omega)$ with $q>2$, and it follows that $g(x,y)$ is H\"older continuous.  Hence so is $u-h_u=\lambda G_0u$.  The H\"older continuity depends on the regularity of the boundary $\partial\Omega$ and can, in principle, be estimated using the quantity $I_q(\Omega)$ on page 337 of~\cite{n2}.

The function $h_u$ is the harmonic extension of the trace $u|_{\partial\Omega}$. The latter is in $\domain(\energy_{\partial\Omega})$ and is therefore $\frac{\log2}{\log3}$-H\"older.  Moreover, the Riemann mapping on $\Omega $ has H\"older continuous extension to $\partial\Omega$, see~\cite{RaimoPalka,RP2,KoskelaReitich}, so $h_u$ is H\"older continuous and the proof is complete. 
\end{proof}

\begin{remark}
The proof of the H\"older continuity of $h_u$ discussed above admits a generalization which does not involve the Riemann mapping and is applicable in any dimension. We refer to~\cite[Chapter 1 Section 2]{Kenig} for the following reasoning; all references are to the definitions and results stated there.  Our domain is class $S$ as in Definition 1.1.20, so a Lipschitz function on $\partial\Omega$ has harmonic extension in a H\"older class $C^{\beta}$ by Lemma~1.2.4.  Taking a sequence of Lipschitz approximations to out boundary data $u|_{\partial\Omega}$ and examining the constants in the proof of Lemma~1.2.4 we discover that these blow-up in manner reflecting the H\"older continuity of $u|_{\partial\Omega}$, and in particular that $h_u$ is H\"older continuous with exponent the worse of the H\"older exponent of the boundary data and the exponent for a Lipschitz function.  A particular case is Remark~1.2.5.  We are indebted to Tatiana Toro for pointing out this reference and sketching the argument.
\end{remark}

A final consequence of Theorem~\ref{thm:ext} which will be significant later is that when high frequency oscillations on $\partial\Omega$ are extended to $\Omega$ they do not penetrate the domain very far and the energy of the extended function is very close to the energy of 
\begin{corollary}\label{cor:oschassmallext}
If the best piecewise harmonic approximations to $f\in\domain(\energy_{\partial\Omega})$ at scale $n$ is the zero function then the extension $g$ is supported within distance $C_1 3^{-n}$ of the boundary and satisfies a seminorm bound of the form
\begin{equation*}
	\int_\Omega |\nabla g|^2 
	\leq C_2 4^{-n} \energy_{\partial\Omega}(f).
	\end{equation*}
\end{corollary}
\begin{proof}
We saw in the construction that the values from  $V_n$ were the only interpolation data for hexagons of sidelength $3^{-m}$, $m\leq n$.  If $f$ is zero at these points then $g\equiv0$ on these hexagons, so its support is within distance at most $C_1 3^{-n}$ from $\partial\Omega$.  Moreover, the terms corresponding to $V_m$, $m\leq n$ in the middle term of the seminorm bound~\eqref{eq:seminormbdforextension} are zero, in which case comparing this to the definition~\eqref{eqn:bdygraphenergy} of the boundary energy we see that the energy scaling factor provides an additional factor of $4^{-n}$.
\end{proof}


\section{Inductive mesh construction and discrete energy forms}
\label{sec:Inductive mesh}
We inductively construct the usual approximations to the closed snowflake domain $\overline\Omega$, along with a triangulating mesh, in the following manner.  The scale $n$ approximation consists of a collection of equilateral triangles with side length $3^{-n}$, the mesh is their edges, and those edges which belong to only one triangle are called boundary edges. When $n=0$ there is exactly one equilateral triangle.  The scale $n+1$ approximation is obtained as shown in Figure~\ref{fig:mesh_1}: each scale $n$ triangle is subdivided into nine triangles as on the left, and triangles are appended to the centers of boundary edges as  on the right. (The figure shows the case $n=0$; at future steps at most two edges of a triangle can be boundary edges.)  The $n=2$ approximation is (f) in Figure~\ref{fig:constructionVertices}.

The mesh is treated as a graph $\Gamma_n$ with vertices $\V_n$.  Note that the the vertices $V_n$ in the graphs in Section~\ref{sec:Koch Snowflake Domain} are a subset of $\V_n$ and indeed the graphs constructed there are simply the boundary subgraphs of the $\Gamma_n$.   Evidently, two vertices $p,q \in\V_n$ are connected by an edge of the mesh if and only if $|p-q|=3^{-n}$, where $|\cdot|$ is the Euclidean norm, in which case we write $p \sim_n q$. The vertices contained in boundary edges of $\Gamma_n$ are  called boundary vertices, all other vertices are interior vertices.

\begin{figure}
{\centering
   \includegraphics[scale=0.5]{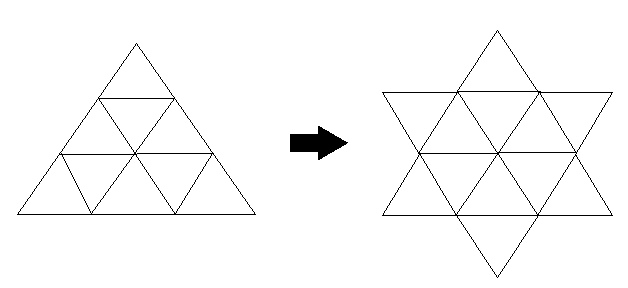} 
    \caption{Mesh construction through scaled equilateral triangles.} 
    \label{fig:mesh_1}
    }
\end{figure}

Let $l(\V_n)=\left\{u \ | \ u:\V_n \to \real \right\}$ and define the graph energy of $u \in l(\V_n)$ by
\begin{equation}
\label{q:graphEnergy}
\energy_n(u) = \sum_{p_1,p_2 \in\V_n} c_n(p_1,p_2)(u(p_1)-u(p_2))^2
\end{equation}
where $c_n(p_1,p_2)$ is the conductance between points $p_1,p_2$, and is given by
\begin{equation}
\label{eq:graphEnergywieghts}
c_n(p_1,p_2) = \begin{cases}
1 &	\textnormal{if $p_1$ and $p_2$ are connected by an interior edge,}\\
4^n & \textnormal{if $p_1$ and $p_2$ are connected by a boundary edge,}\\
0 & \textnormal{if $p_1$ and $p_2$ are not connected by an edge}.
\end{cases} 
\end{equation}

Note that for boundary vertices both~\eqref{eq:graphEnergywieghts} and~\eqref{eq:discreteMeasure} agree with those in Section~\ref{sec:Koch Snowflake Domain}  and therefore the restriction of~\eqref{q:graphEnergy} to edges in the boundary of $\Gamma_n$ coincides with the energy defined in~\eqref{eqn:bdygraphenergy}.  Moreover, the terms corresponding to edges in $\Omega$ have weight $1$, and it follows from Lemma~\ref{lem:energyisdiffssquared} that for a function $g$ which is piecewise linear on equilateral triangles the sum of squared differences of the function over edges in the triangulation is a constant multiple of the Dirichlet energy $\int_{\Omega} |\nabla g|^2$.  After computing this constant and verifying that our triangulation by $\Gamma_n$ is obtained from the $n^\text{th}$ triangulation in the extension of Theorem~\ref{thm:ext} simply by dividing all triangles of sidelength larger than $3^{-n}$ into subtriangles of sidelength $3^{-n}$ the following consequence is immediate.

\begin{theorem}
\label{propositionDiscreteEnergy}
If $u\in H^1_{\partial\Omega}$ is the extension of $f\in\domain(\energy_{\partial\Omega})$ by the extension procedure of Theorem~\ref{thm:ext} then $\energy (u) = \lim_{n \to \infty} \energy_n(u)$.
\end{theorem}
\begin{remark}
$\energy_n(u)$ should be thought of as the graph energy of the restriction of $u$ to $\V_n$.
\end{remark}

We now introduce a measure on $\V_n$ by
\begin{equation}
\label{eq:discreteMeasure}
m_n(p)=
\begin{cases}
9^{-n} & \text{if $p$ is an interior vertex}  \\ 4^{-n} & \text{if $p$ is a boundary vertex}.
\end{cases}
\end{equation}
This equips $l(\V_n)$ with an inner product 
\begin{align*}
\label{eq:discreteInnerProduct}
	\langle \cdot , \cdot\rangle_n &: l(\V_n) \times l(\V_n)  \to  \real \\
	\langle u,v \rangle_n &=\sum_{p \in \V_n} m_n(p) u(p)v(p),
\end{align*}
and we can then define a discrete Laplacian $L_n$ so that $\energy_n(u,v)=- \langle L_n u,v \rangle_n$ as follows.
\begin{definition}
\label{graphLaplace}
Let $u \in l(V_n)$ and $p \in \V_n$. The ($n$-level) discrete Laplacian of $u$ at $p$ is
\begin{equation*}
L_n u (p) = \frac{2}{m_n(p)}\sum_{q \in V_n} c_n(p,q)(u(p)-u(q)),
\end{equation*}
with conductances $c_n(p_1,p_2)$ from~\eqref{eq:graphEnergywieghts} and vertex measure $m_n(p)$ from~\eqref{eq:discreteMeasure}.
\end{definition}



\section{Numerical Results}
\label{sec:numericalResults}
In order to investigate the spectral properties of the Dirichlet form $\energy$ we constructed the discrete Laplacian matrix $L_n$ from Definition~\ref{graphLaplace} and solved $L_n\phi=\lambda \phi$ numerically.  We compared the resulting eigenvalues and eigenvectors with those of the Dirichlet Laplacian, which we denote $\tilde{L}_n$.  In this section we almost exclusively present data for $n=4$, and denote the $j^\text{th}$ eigenvector and eigenvalue of $L_4$ by $\phi_j$ and $\lambda_j$, where $\lambda_j\leq\lambda_{j+1}$ for all $j$ and eigenvalues are repeated according to their multiplicity.  We label the eigenvectors $\tilde{\phi}_j$ and $\tilde{\lambda}_j$ in the same manner.

In order to justify the above numerical approach to investigating the spectrum of $\energy$ we should like to prove a result of the following kind.  We believe this ought to be possible by methods analogous to those in~\cite{l1,l2,l3,EvansE}. 
\begin{conjecture}
Our discrete and finite element approximations to eigenfunctions of $\energy$ converge uniformly, and hence the spectrum of $L_n$ converges to that of $L$.
\end{conjecture}

\subsection{Algorithm and  Implementation}
\begin{figure}[th]
{\centering
\begin{tabular}{cccc}
    \includegraphics[scale=0.31]{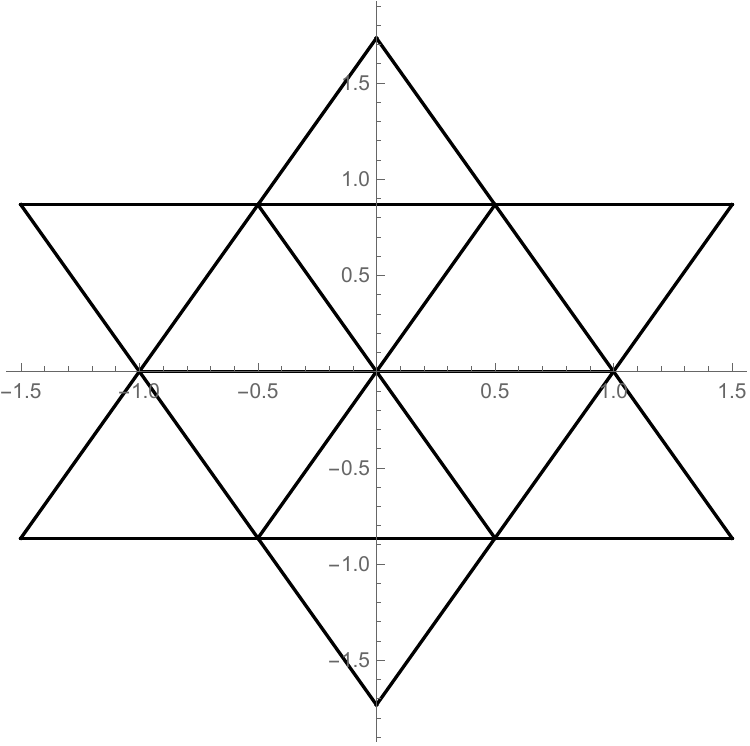} &
    \includegraphics[scale=0.31]{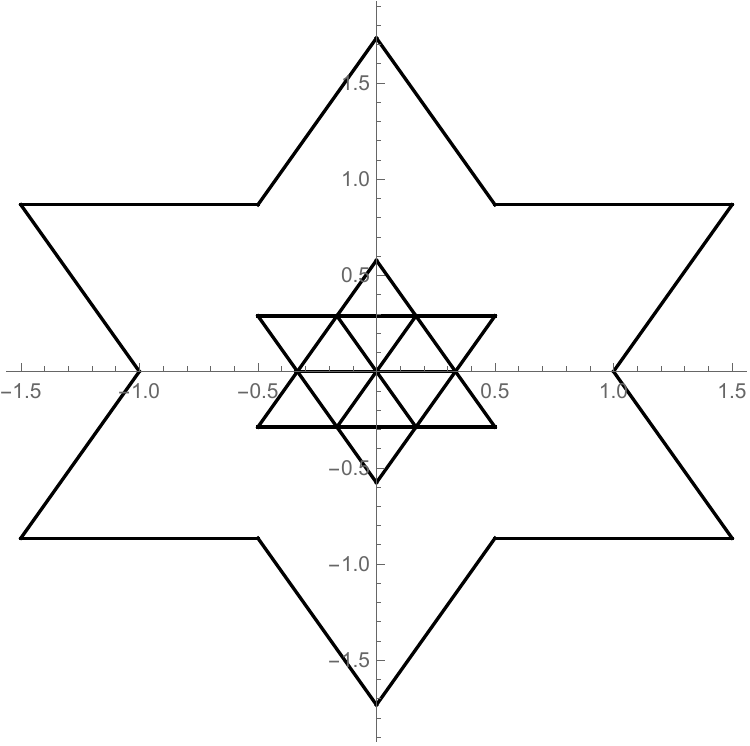} &
    \includegraphics[scale=0.31]{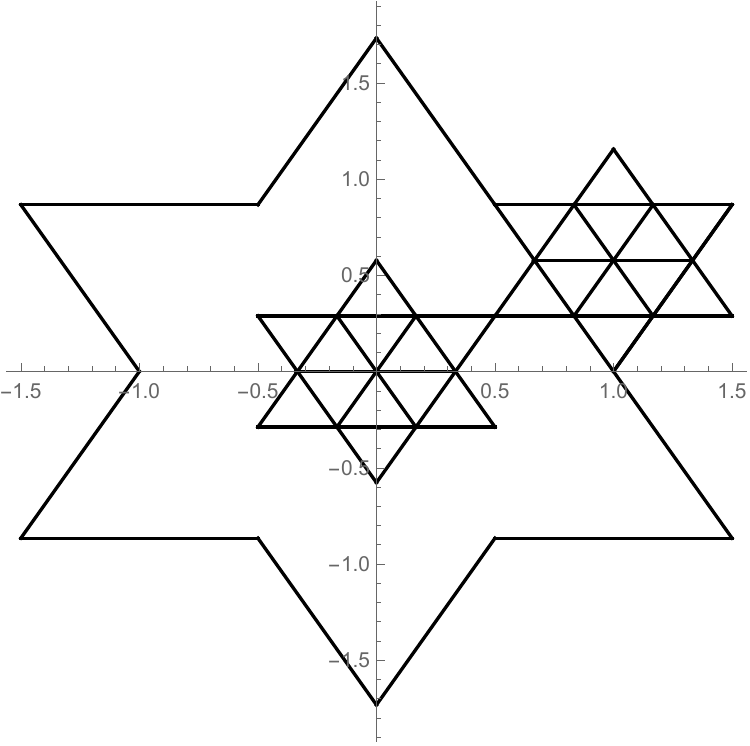} \\
(a) & (b) & (c) 
\\
    \includegraphics[scale=0.31]{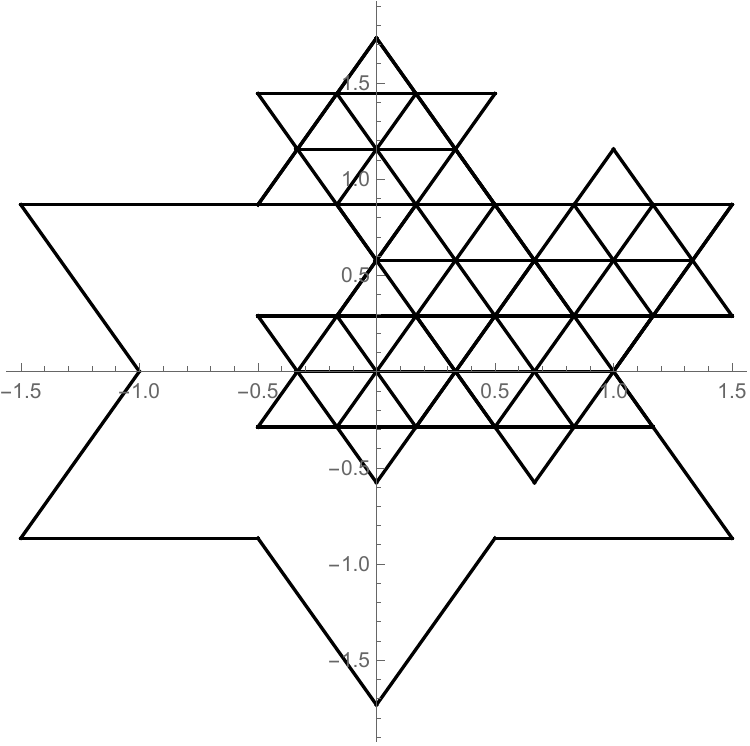} & 
    \includegraphics[scale=0.31]{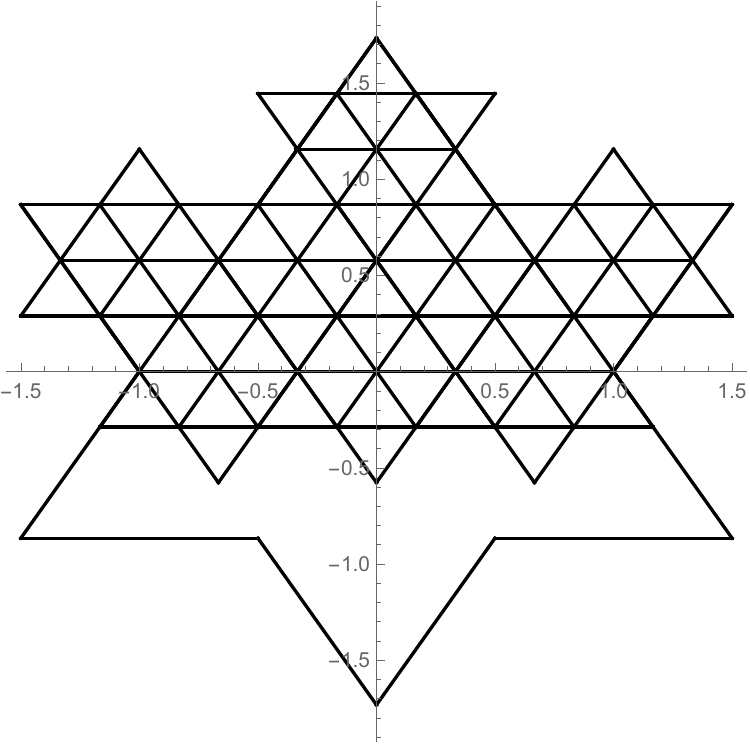} & 
    \includegraphics[scale=0.31]{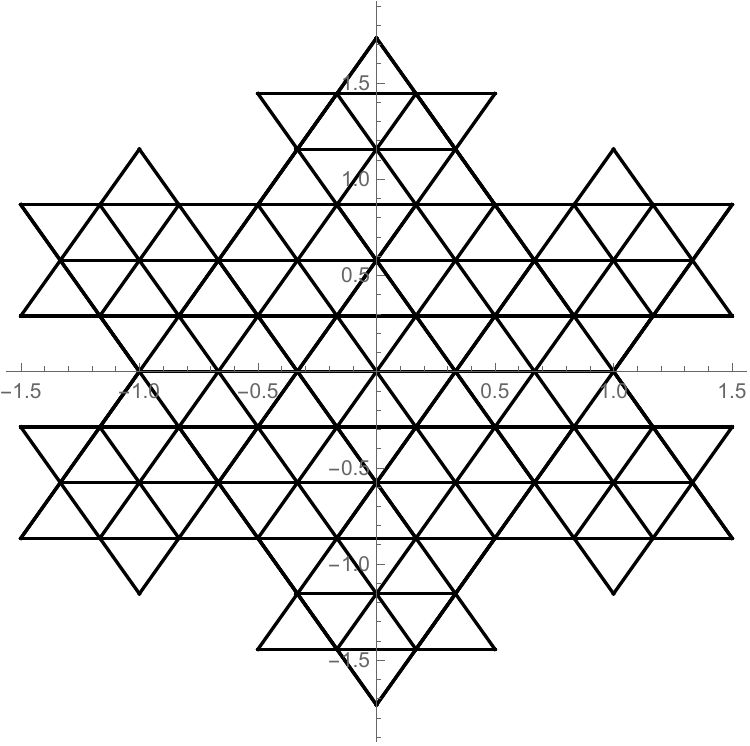} \\
(d) & (e) & (f) 
\\
\end{tabular}
\caption{Algorithm to generate the vertices of the graph $\Gamma_n$}
\label{fig:constructionVertices}
}
\end{figure}

We constructed $L_n$ by an iterative procedure written in the programming language Python and then computed and graphed the eigenvalues and eigenfunctions of $L_n$ using Mathematica.  We also computed the matrix $\tilde{L}_n$ of the Dirichlet Laplacian so that comparison of $L_n$ and $\tilde{L}_n$ could be used to identify effects of the boundary Dirichlet form.

The construction of $L_n$ involved constructing the mesh graphs $\Gamma_n$.  The vertices of $\Gamma_1$, which is shown in Figure~\ref{fig:constructionVertices}(a),  were hard-coded. Construction of the vertices of $\Gamma_2$ was achieved by the steps (b)--(f) of Figure~\ref{fig:constructionVertices}: specifically, $\Gamma_1$ was scaled by $\frac13$, translated to cover the region in the first quadrant and reflected to the other quadrants. Duplicate vertices were deleted and the process repeated to construct the vertices of $\Gamma_n$ for $n=3$ and $n=4$.

Construction of $L_n$ is equivalent to finding the vertex adjacencies of $\Gamma_n$; an elementary way to do this is to find pairs of points separated by distance $3^{-n}$.  However it was not efficient to identify boundary vertices in this manner, and we used KD trees to improve run times for this aspect of the problem. With the adjacencies and boundary edges identified it is easy to weight these and construct the  matrix $L_n$. The matrix $\tilde{L}_n$ of the Dirichlet Laplacian is obtained by deleting the rows and columns corresponding to boundary vertices.

\subsection{The eigenvalue counting function}
The eigenvalue counting function for a general symmetric matrix $M$ is defined by
\begin{equation*}
N_M(x) := \# \{\lambda \leq x \ | \ \lambda \text{ is an eigenvalue of } M \}
\end{equation*}
where eigenvalues are counted with multiplicity.
This generalizes to operators with discrete spectrum in $[0,\infty)$.
It is well-known that if the operator in question is the (positive) Laplacian $\Delta$ on a Euclidean domain then $N_\Delta$ encodes a great deal of geometric information: for example, the Weyl law says that it grows as $C_d x^{d/2}$, where $d$ is the dimension and $C_d$ encodes the volume of the domain; more precise asymptotics can be obtained for certain classes of domains.

\begin{figure}
{\centering
\includegraphics[width=.48\textwidth]{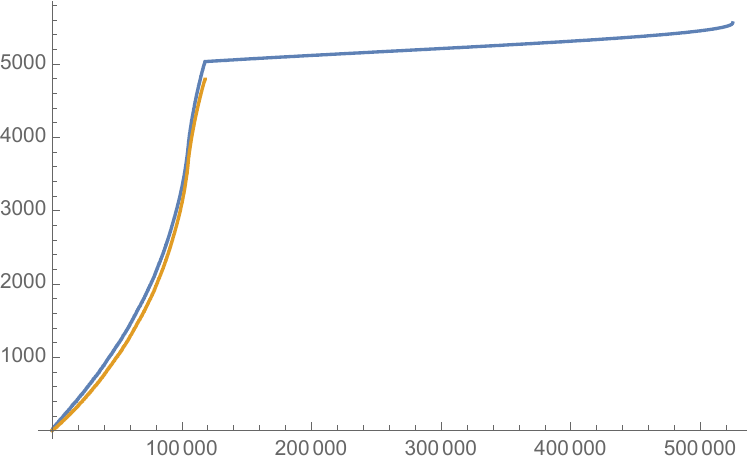} \hfill
\includegraphics[width=.48\textwidth]{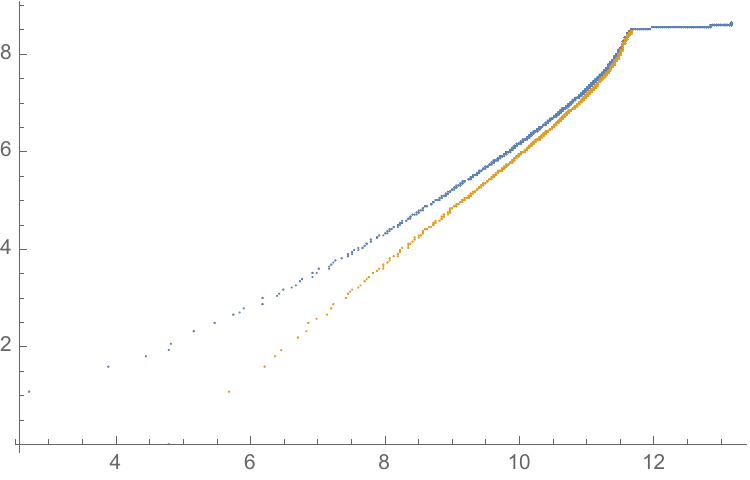}
\caption[Left: Eigenvalue counting functions $N_{L_4}$ (blue) and $N_{\tilde{L}_4}$ (orange) Right: Log-Log plots of  $N_{L_4}$ (blue) and $N_{\tilde{L}_4}$ (orange)]{\begin{tabular}[t]{ @{} l @{\ } l @{}}
    Left: & Eigenvalue counting functions  $N_{L_4}$ (blue) and $N_{\tilde{L}_4}$ (orange) \\
    Right: & Log-Log plots of  $N_{L_4}$ (blue) and $N_{\tilde{L}_4}$ (orange)  
  \end{tabular}}
\label{fig:EigenvalueCount}
}
\end{figure}

The eigenvalue counting functions of $L_n$ and the Dirichlet Laplacian $\tilde{L}_n$ for the $n=4$ approximation are displayed in Figure~\ref{fig:EigenvalueCount}, along with Log-Log plots that identify their growth behavior.  It is immediately apparent that $N_{L_4}$ has two regimes, a low eigenvalue regime in which it fairly closely tracks the behavior of $N_{\tilde{L}_4}$, though with a lower initial rate of growth, and a high eigenvalue regime in which its growth is entirely different.  Moreover, there is an additional feature which is not apparent from the graphs.  The largest eigenvalue of $\tilde{L}_4$ is $\tilde{\lambda}_{4789}=118039.37$, and the regime-change point on the $L_4$ graph occurs at the almost identical value $\lambda_{5028} = 118038.02$.  Graphs of the corresponding eigenfunctions are also very similar; they are both highly oscillatory, so their graphs look multi-valued, but the similarity in the macroscopic profile of the oscillations is readily apparent in Figure~\ref{fig:lastEigvecs2Kink}.

\begin{figure}
{\centering
\begin{tabular}{cc}
\includegraphics[width=0.5\textwidth]{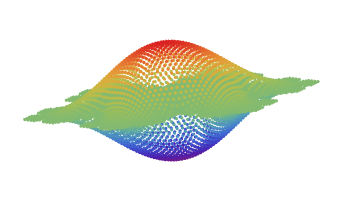}&
\includegraphics[width=0.5\textwidth]{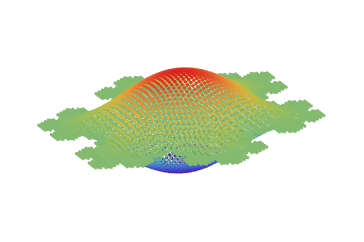}
\\
$\phi_{5028}$,  $\lambda_{5028}=118038.02$ & $\tilde{\phi}_{4789}$, $\tilde{\lambda}_{4789} = 118039.37$
\end{tabular}
\caption{The eigenvector $\phi_{5028}$ of $L_n$ at the regime change $\lambda\sim118038.5$ is qualitatively similar to the last Dirichlet eigenvector $\tilde{\phi}_{4789}$}
\label{fig:lastEigvecs2Kink}
}
\end{figure}

\subsection{Eigenvectors and eigenvalues in the low eigenvalue regime}\label{ssec:loweval}
The eigenstructure of $L_n$ in the low eigenvalue regime holds few surprises.  Comparing it to that for $\tilde{L}_n$ we found that the boundary Dirichlet form permitted many additional low-frequency configurations, some of which are shown in Figure~\ref{fig:first3eigvecs}.  It is apparent from the counting function plots in Figure~\ref{fig:EigenvalueCount} that this difference in the growth of $N_{L_n}$ and $N_{\tilde{L}_n}$ decreases as $\lambda$ increases.

\begin{figure}[h!]
{\centering
\begin{tabular}{ccc}
\includegraphics[width=0.3\textwidth]{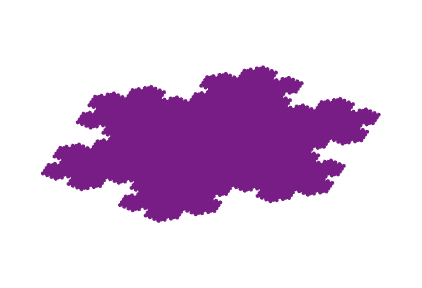}&\includegraphics[width=0.3\textwidth]{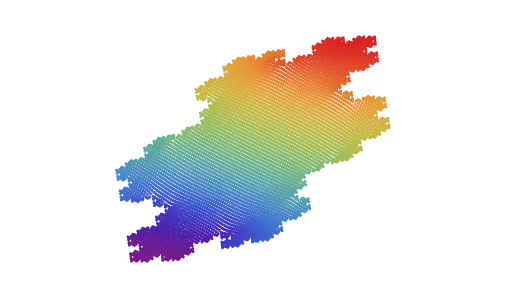}&\includegraphics[width=0.3\textwidth]{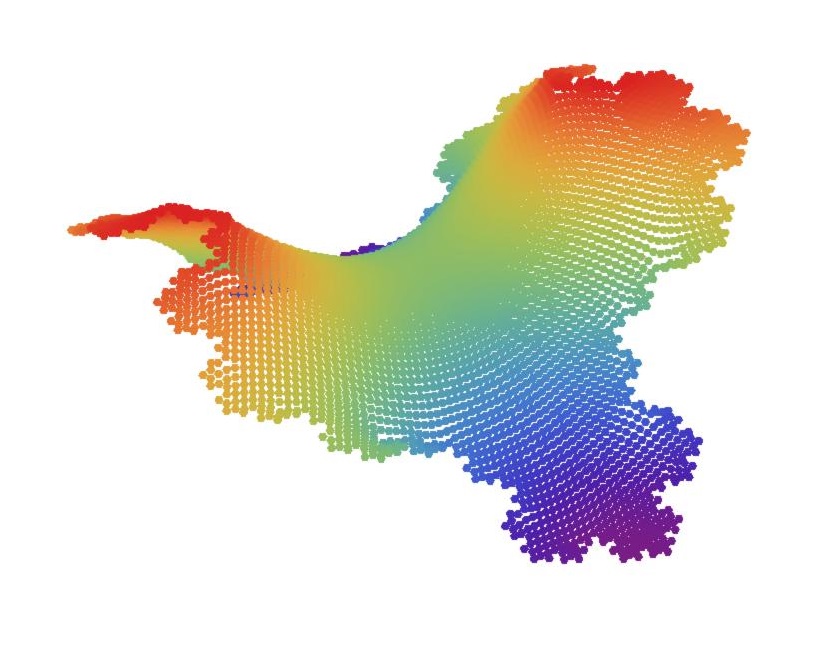}\\
$\phi_1$, $\lambda_1=0$ & $\phi_2$, $\lambda_2=15.1$ & $\phi_4$, $\lambda_4=48.1$\\
\includegraphics[width=0.3\textwidth]{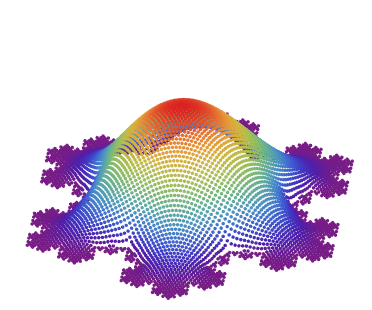} & \includegraphics[width=0.3\textwidth]{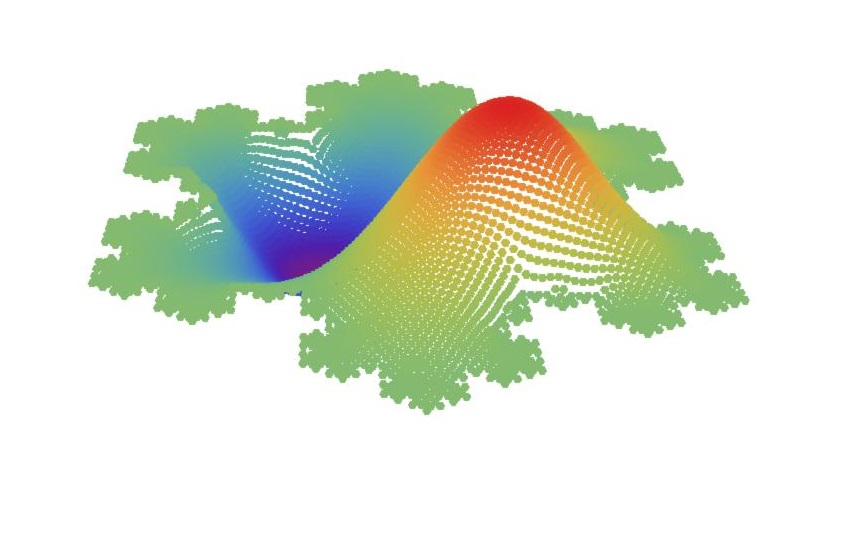} & \includegraphics[width=0.3\textwidth]{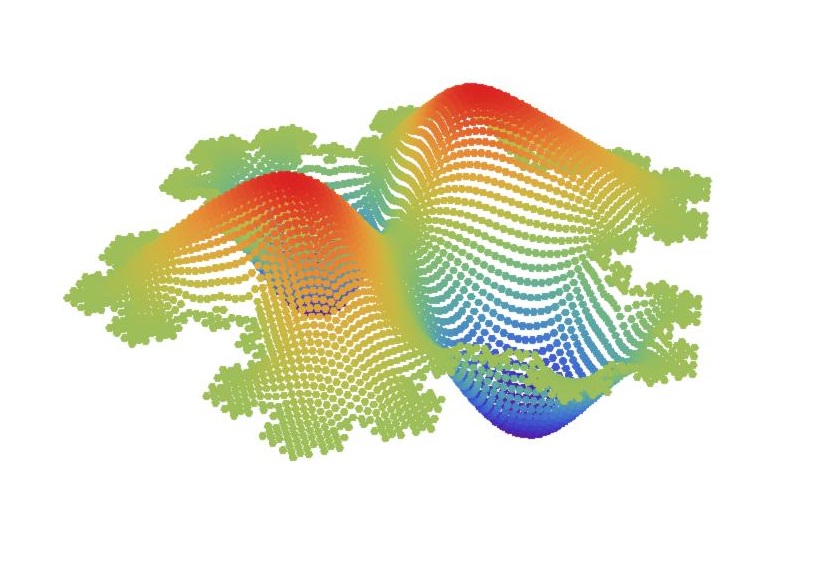}\\
$\tilde{\phi}_1$, $\tilde{\lambda}_1=118.8$ & $\tilde{\phi}_2$, $\tilde{\lambda}_2=294.5$ & $\tilde{\phi}_4$, $\tilde{\lambda}_4=499.8$
\end{tabular}
\caption{Selected eigenvectors of $L_4$ (above) and $\tilde{L}_4$ (below).}
\label{fig:first3eigvecs}
}
\end{figure}

The eigenvalues less than $\lambda=1100$ for each operator are in the Table~\ref{table}.  It is evident that for both operators some eigenvalues occur with multiplicity $1$ and some with multiplicity $2$. This is true throughout the spectrum (not just in the low frequency regime) and is explained by symmetries of $\overline{\Omega}$; this may be verified by the argument given in~\cite{MR1412217}.  At higher frequencies these symmetries are difficult to see from the graphs but readily apparent in contour plots, as illustrated in Figure~\ref{fig:contoureigvecs}.

\begin{table}
{\centering
\begin{tabular*}{\textwidth}{@{\extracolsep{\fill}}lclclclclc|lclc}
\hline
\multicolumn{10}{c}{Eigenvalues of $L_4$} & \multicolumn{4}{c}{Eigenvalues of $\tilde{L}_4$\rule{0pt}{1.2em}}\\
\hline
   j & $\lambda_j$ & j & $\lambda_j$ & j & $\lambda_j$ & j & $\lambda_j$ & j & $\lambda_j$ & j & $\tilde{\lambda}_j$ & j & $\tilde{\lambda}_j$\rule{0pt}{1.1em} \\
   1  &  0       & 8  &  125.4   &15  &  344.8  &22  &  617.9 &29  &  880.9  & 1 & 118.8  &8 &  822.7  \\
   2  &  15.1  & 9  &  171.6   &16  &  363.6  &23  &  651.6 &30  &  880.9   &2 &  294.5  &9 &  822.7  \\
   3  &  15.1  &10  &  171.6   &17  &  482.0  &24  &  651.6 &31  &  1007.2 &3 &  294.5 &10 & 941.7  \\
   4  &  48.1  & 11  &  238.5  &18  &  482.0  &25  &  743.8 &32  &  1014.9 &4 &  499.8 &11 & 950.5  \\
   5  &  48.1  & 12  &  238.5  &19  &  490.9  &26  &  787.9 &33  &  1014.9 &5 &  499.8 &12 & 950.5  \\
   6  &  85.1  & 13  &  313.0  &20  &  490.9  &27  &  851.2 &34  &  1098.6 &6 &  575.1 &13 & 1084.6 \\
   7  &  119.2 &14  &  313.0  &21  &  609.5  &28  &  851.2 &      &             &7 &  630.5 &    & \\
   \hline
\end{tabular*}
}\caption{Eigenvalues of $L_4$ and  $\tilde{L}_4$.}\label{table}\end{table}

\begin{figure}
{\centering
\begin{tabular}{ccc}
\includegraphics[width=.23025\textwidth]{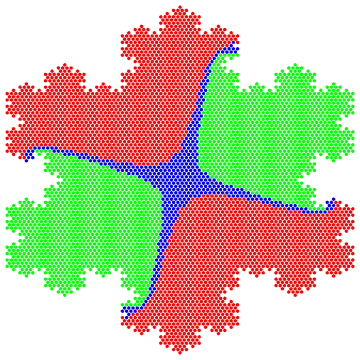}&
\includegraphics[width=0.23025\textwidth]{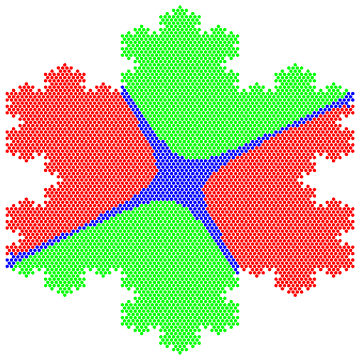}& 
\includegraphics[width=0.23025\textwidth]{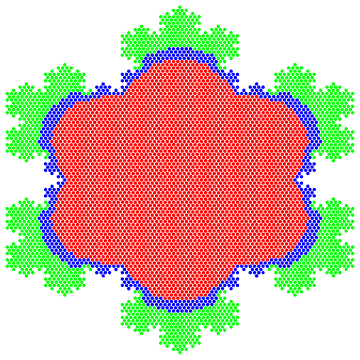}\\
$\phi_4$, $\lambda_4=48.1$ & $\phi_5$, $\lambda_5=48.1$ & $\phi_8$, $\lambda_8=125.4$ \\
\includegraphics[width=0.23025\textwidth]{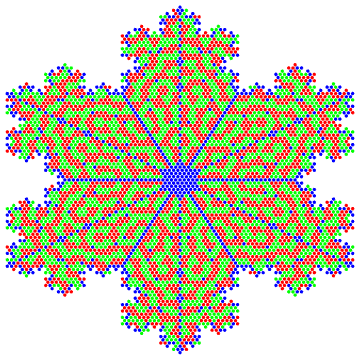}&
\includegraphics[width=0.23025\textwidth]{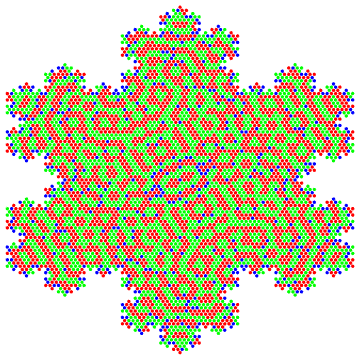}&
\includegraphics[width=0.23025\textwidth]{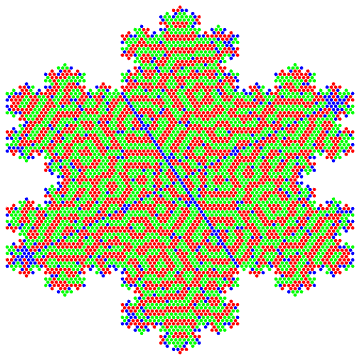} \\
$\phi_{1153}$, $\lambda_{1153}=49965.7$ & $\phi_{1161}$, $\lambda_{1161}=50188.8$ & $\phi_{1162}$,  $\lambda_{1162}=50188.83$
\end{tabular}
\caption{Contour plots  illustrating symmetries of some eigenvectors of $L_4$. Blue indicates $|\phi|\leq\epsilon$, red is $\phi>\epsilon$ and green  is $\phi<\epsilon$, for $\epsilon=0.01$.}
\label{fig:contoureigvecs}
}
\end{figure}

One striking observation about the spectra of  $L_n$ and $\tilde{L}_n$ was the existence of eigenvectors and eigenvalues that were very similar in both.  An example is shown in Figure~\ref{fig:eigvecs}, where it is apparent that $\phi_{34}$ with eigenvalue $\lambda_{34}=1098.6$ is very similar to $\tilde{\phi}_{13}$ with eigenvalue $\tilde{\lambda}_{13}=1084.6$. Such ``pairs'' of eigenvectors, one from $L_n$ and one from $\tilde{L}_n$, both with the same symmetries and similar eigenvalues, were found throughout the low eigenvalue regime. One possible explanation is that there might be Dirichlet eigenfunctions with symmetries that imply they are close to being eigenfunctions of $L$.

To see this, let us try to compute how far a Dirichlet eigenfunction might be from being an eigenfunction of $L$.  Making a computation like that in Theorem~\ref{thm:efnsHolder} we see that a Dirichlet eigenfunction is $u\in H^1_0(\Omega)$ such that  $\energy_\Omega(u,v)=\lambda\langle u,v\rangle_{L^2(\Omega)}$ for all $v\in H^1_0(\Omega)$.  A general element of $\domain(\energy)$ can be written as $v+h_f$, where $f\in\domain(\energy_{\partial\Omega})$ denotes the boundary values and $h_f$ is harmonic. Since $u\in H^1_0(\Omega)$ has zero boundary trace we have
\begin{equation*}
	\energy(u,v+h_f)
	=\energy_\Omega (u,v)
	=\lambda\langle u,v+h_f\rangle_{L^2(\overline\Omega,dm)} - \lambda\langle u,h_f\rangle_{L^2(\Omega)}
	\end{equation*}
so $u$ would also be an eigenfunction of $L$ if $\langle u,h_f\rangle_{L^2(\Omega)}=0$ for all $h_f$.  
Therefore, we would expect to find an eigenfunction and eigenvalue near $(u,\lambda)$ if $\langle u,h_f\rangle_{L^2(\Omega)}$ is small for all $f\in\domain(\energy_{\partial\Omega})$ normalized so that  $ \energy_{\partial\Omega}(f)+\|f\|^2_{L^2 ({\partial\Omega})} =1$.  In consideration of Corollary~\ref{cor:oschassmallext} we might expect  highly oscillatory $f$ to produce $h_f$ with support close to the boundary, where $u$ is small because it has Dirichlet boundary values.  This heuristic suggests that we need only consider slowly varying $f$, and, moreover, that if the symmetries of $u$ are such that it is nearly orthogonal to a large enough subspace of harmonic extensions of slowly varying boundary value functions, then $u$ should be very close to being an eigenfunction of $L$ as well as $\tilde{L}$.  We do not know if this heuristic is the correct explanation for the observed phenomenon, or have estimates that would make the explanation precise.

\begin{figure}
{\centering
\begin{tabular}{cc}
\includegraphics[width=0.5\textwidth]{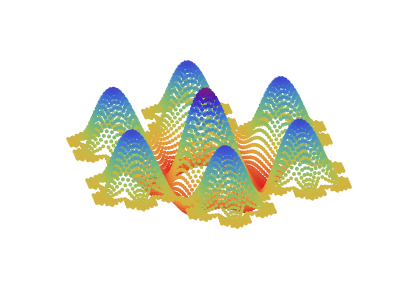}&
\includegraphics[width=0.5\textwidth]{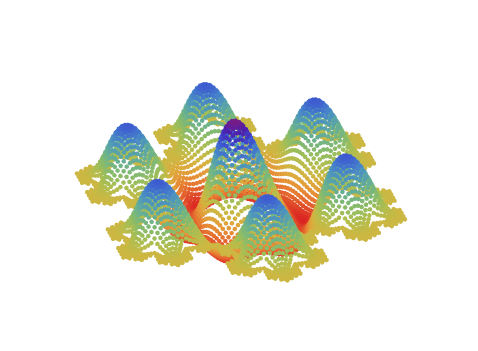}\\
$\phi_{34}$, $\lambda_{34}=1098.6$ & $\tilde{\phi}_{13}$, $\tilde{\lambda}_{13}=1084.6$.
\end{tabular}
\caption{A ``pair'' of eigenvectors, one for $L_4$ and one for $\tilde{L}_4$, with similar symmetry structure and eigenvalues.}
\label{fig:eigvecs}
}
\end{figure}


\subsection{Localization in the high eigenvalue regime}
The high eigenvalue regime contains the eigenvalues of $L_n$ that are larger than the largest Dirichlet eigenvalue.   In this regime we see a dramatic transition from eigenfunctions that are supported throughout the domain to eigenfunctions localized near the boundary.
Figure~\ref{fig:eiglocaldynamic} illustrates this change using contour plots for eigenfunctions of $L_4$: the first row of images are of eigenfunctions just inside the regime, while the second row shows the progression of localization with increasing frequency.  The last image is the highest eigenvalue eigenfunction of $L_4$. Note that the onset of the localization phenomenon is rapid, occurring across just a few eigenfunctions, and that  continues to sharpen as the frequency increases, but at a decreasing rate.  

Several phenomena in physics involve the localization of eigenmodes (e.g.\ Anderson localization due to a random potential in the Schr\"odinger equation).  A nice survey from a mathematical perspective is in~\cite{MR3038118}.   Some involve low frequency modes (e.g.\ weak localization due to irregular or complex geometry of the domain) and some involve high frequency modes (e.g.\ whispering-gallery modes, bouncing ball modes and focusing modes).  The type that looks most similar to those observed for $L_n$ are the whispering-gallery modes.  These are well understood for convex domains~\cite{KELLER196024}, but despite the fact that for many years  similar modes have been observed in non-convex domains, including domains with pre-fractal boundaries, our understanding of why these occur is incomplete. One approach to high-frequency localized eigenmodes uses quantum billiards, but studying quantum billiards on sets with fractal boundaries seems to be a difficult problem~\cite{MR3203866}. The only literature we are aware of for the Koch snowflake domain is~\cite{MR2731085,MR3021377}, and it does not include modes of the type we see for $L_n$.

\begin{figure}
{\centering
\begin{tabular}{ccc}
\includegraphics[width=0.2302\textwidth]{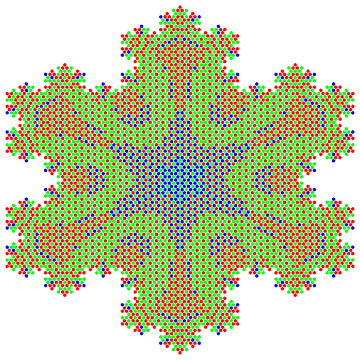}&
\includegraphics[width=0.2302\textwidth]{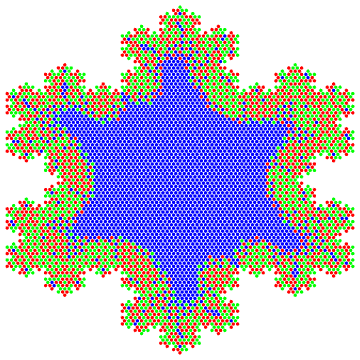}& \includegraphics[width=0.2302\textwidth]{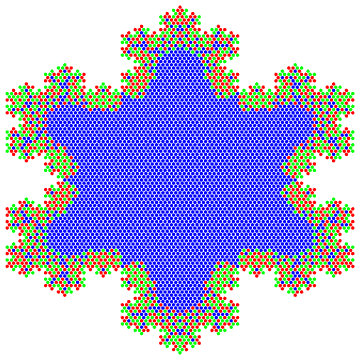}\\
$\phi_{5030}$, $\lambda_{5030}= 118048.66$ & $\phi_{5031}$, $\lambda_{5031}= 119678.65$ &
$\phi_{5033}$, $\lambda_{5033}=121460.72$ \\
\includegraphics[width=0.2302\textwidth]{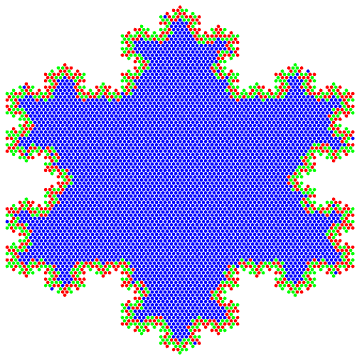} &
\includegraphics[width=0.2302\textwidth]{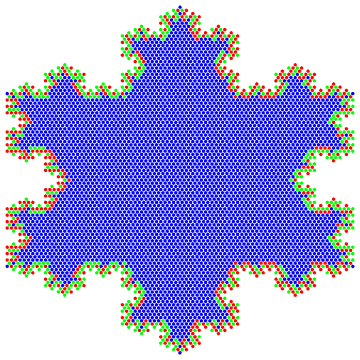} &\includegraphics[width=0.2302\textwidth]{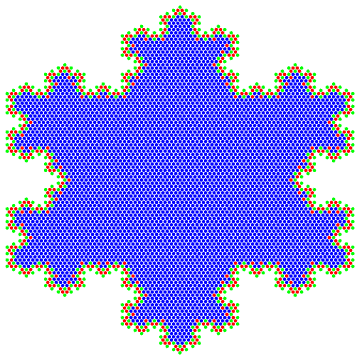} \\
$\phi_{5100}$, $\lambda_{5100}= 185367.41$ & 
$\phi_{5200}$, $\lambda_{5200}= 291364.38$ & $\phi_{5557}$, $\lambda_{5557}= 524999.69$
\end{tabular}
\caption{Contour plots of $L_4$ eigenvectors illustrating localization. Blue indicates $|\phi|\leq\epsilon$, red is $\phi>\epsilon$ and green  is $\phi<-\epsilon$, for $\epsilon=0.01$.}
\label{fig:eiglocaldynamic}
}
\end{figure}

We believe that in our setting the high frequency boundary localizations are not due to a whispering gallery effect, but rather to the relatively weak coupling between the boundary energy and the domain energy.  We can see this, at least heuristically, by repeating the sort of calculation done at the end of subsection~\ref{ssec:loweval} above, but this time considering how close a high frequency eigenfunction of the boundary energy $\energy_{\partial\Omega}$ might be to an eigenfunction of $L$.  Suppose $u\in\domain(\energy_{\partial\Omega})$ satisfies $\energy_{\partial\Omega}(u,v)=\lambda\langle u,v\rangle_{L^2(\Omega,\mu)}$ for all $v\in \domain(\energy_{\partial\Omega})$. If we extend $u$ and $v$ harmonically to $\Omega$, obtaining $h_u$ and $h_v$ respectively, and add an arbitrary $\tilde{v}\in H^1_0(\Omega)$ to $h_v$ so that $h_v+\tilde{v}$ ranges over $\domain(\energy)$, we obtain
\begin{equation*}
	\energy(h_u,h_v+\tilde{v})
	= \energy_{\partial\Omega}(u,v) + \energy_\Omega(h_u,h_v) +\energy_\Omega(h_u,\tilde{v})
	\end{equation*}
Now if $u$ is of high frequency it is highly oscillatory and it is not unreasonable to expect that its piecewise harmonic approximation at large scales is small.  In consideration of Corollary~\ref{cor:oschassmallext} one might expect that then $\energy_\Omega(h_u,h_u)$ is small, and therefore $\energy(h_u,h_v+\tilde{v})\approx\energy_{\partial\Omega}(u,v)$.  Moreover, we have
\begin{equation*}
	\langle h_u,h_v+\tilde{v} \rangle_{L^2(\overline\Omega,dm)}
	= \langle h_u,h_v+\tilde{v} \rangle_{L^2(\Omega,d\mathcal{L}^2)}+ \langle u ,v \rangle_{L^2(\partial\Omega,d\mu)}
	\end{equation*}
and again we see from Corollary~\ref{cor:oschassmallext} that one should expect $\|h_u\|_{L^2(\Omega,d\mathcal{L}^2)}$ to be small if $u$ is high frequency, leading to $\langle h_u,h_v+\tilde{v} \rangle_{L^2(\overline\Omega,dm)}\approx  \langle u ,v \rangle_{L^2(\partial\Omega,d\mu)}$.  Combining this with the corresponding statement for the energies and the assumption that $u$ was an eigenfunction of $\energy_{\partial\Omega}$ we have
\begin{equation*}
	\energy(h_u,h_v+\tilde{v})\approx \energy_{\partial\Omega}(u,v)=\lambda\langle u,v\rangle_{L^2(\Omega,\mu)}
	\approx \langle h_u,h_v+\tilde{v} \rangle_{L^2(\overline\Omega,dm)}
	\end{equation*}
which, if exactly true, would say that $h_u$ was an eigenfunction of $\energy$ with eigenvalue $\lambda$.  This argument is consistent with what we observe, especially the fact that oscillations of the most localized of the high frequency eigenvectors of $L_n$ appear to have wavelength approximately the length of an edge of $\Gamma_n$ and are indeed localized very closely to the boundary curve, as seen in Figure~\ref{fig:lasteigvec}. However we do not have precise arguments or estimates to justify this heuristic reasoning. A less regular looking eigenfunction localized on the boundary is shown in Figure~\ref{eigenfunction5550}.

{\centering
	\begin{figure}
		\includegraphics[width=0.5\textwidth]{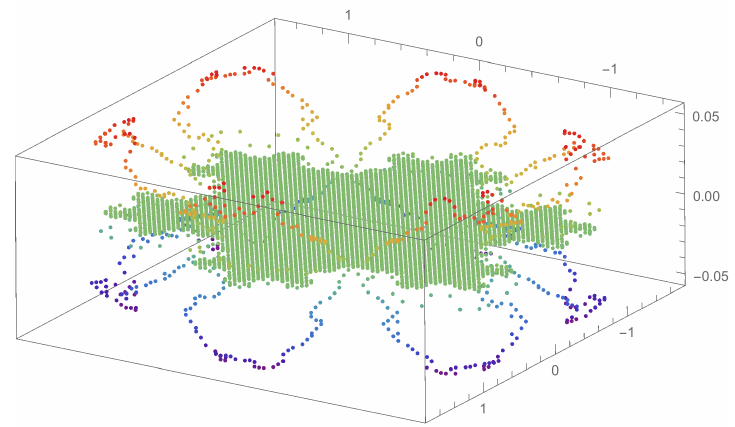}
		\caption{The 5550th $L_4$ eigenfunction, $\phi_{5550}$.   }\label{eigenfunction5550}
	\end{figure}
}  

\begin{figure}
{\centering
\includegraphics[width=0.5\textwidth]{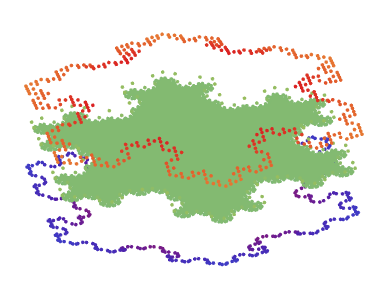}
\caption{ The last $L_4$ eigenfunction, $\phi_{5557}$ with $\lambda_{5557}= 524999.69$ oscillates so rapidly and is so small away from $\partial\Omega$ that its graph appears to split into copies of $\partial\Omega$ at a positive and negative height and $\Omega$ at height zero.}
\label{fig:lasteigvec}
}
\end{figure}

\section{A landscape approach to high frequency localization}
\label{FilocheMayboroda Argument}
Although many aspects of wave localization problems remain open, pioneering work of Filoche and Mayboroda~\cite{filoche1,MR2990982,ADFJM} has led to a great deal of progress on low-frequency localization problems.  For a suitable elliptic differential operator $\mathcal{L}$ on a bounded open set $U\subset\mathbb{R}^n$, they introduce a ``landscape'' function $u:U\to\mathbb{R}$ by $u=\int_U |G(x,y)|$, where $G$ is the Dirichlet Green's function, and show that an $L^\infty$ normalized Dirichlet eigenfunction $\phi$ with $\mathcal{L}\phi=\lambda\phi$ satisfies $|\phi(x)|\leq\lambda u(x)$.  In consequence, if there is a region where $u$ is small  (a valley) that is enclosed by a set where $u$ is large (a set of peaks), then a low frequency eigenfunction that is non-zero in the valley must be confined to that valley.  This reduces the problem of low frequency eigenfunction localization to studying the structure of the landscape function, and especially its level sets.  The latter involves a great deal of sophisticated mathematics, but is numerically very simple to implement.

Our purpose here is to write a high frequency variant of the argument of Filoche and Mayboroda, and to see that its numerical implementation predicts the high frequency localization seen in our data.  
We do so by introducing a high frequency landscape function associated to a linear map $L$.
\begin{definition}
For $L:\mathbb{R}^d\to\mathbb{R}^d$ is a linear map expressed as a matrix $\bigl[L_{ij}\bigr]$ with respect to the standard basis, we let $|L|:\mathbb{R}^d\to\mathbb{R}^d$ denote the linear map with matrix   $\bigl[|L_{ij}|\bigr]$ and define the high frequency landscape vector to be $u=|L|(1,\cdots,1)^{T}$. Thus $u=(u_1,\cdots,u_d)$ with $u_i=\sum_j |L_{ij}|$.
\end{definition}
We remark that the high frequency landscape depends on $|L|$ whereas the low frequency landscape function depended on $|G|=|\mathcal{L}^{-1}|$. The following elementary theorem is the analogue of the Filoche-Mayboroda bound using the high frequency landscape function.
\begin{theorem}
\label{landscapeTheorem}
Let $\phi = (\phi_1, \ldots , \phi_d)^T \in \real^d$ be an eigenvector of $L$ corresponding to an eigenvalue $\lambda>0$ normalized such that $\max_{1 \leq i \leq d} |\phi_i|=1$. Then for each $i=1,\cdots,d$,
\begin{equation}\label{eq:highfreqest}
\phi_i  \leq \frac{1}{\lambda} u_i.
\end{equation}
\end{theorem}
\begin{proof}
Since $\lambda\phi = L \phi$ we may compute using the standard basis $\{e_i\}$
\begin{equation*}
\lambda\phi_i 
= \langle e_i,L\phi \rangle
=\sum_j \phi_j \langle e_i,L\phi_j \rangle
=\sum_j \phi_j L_{ij}
\leq \sum_j |L_ij|
= u_i
\end{equation*}
where the inequality used the normalization that $|\phi_j|\leq1$ for all $j$.
\end{proof}

\begin{figure}
{\centering
\includegraphics[width=0.6\textwidth]{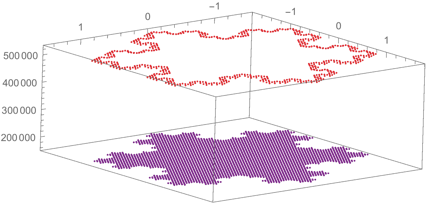}
\caption{The high frequency landscape for $L_4$}
\label{fig:landscape}
}
\end{figure}

The result of numerical computation of the high frequency landscape function for $L=L_4$, our discrete Laplacian, is shown in Figure~\ref{fig:landscape}.  It appears to be constant on $\Omega$ and on $\partial\Omega$, but in fact we find that it attains two values on $\partial\Omega$.  This can be computed directly from the formula in Definition~\ref{graphLaplace}.
\begin{lemma}
The high frequency landscape function $u$ for $L_n$ has constant value $8\cdot3^{2n+1}$ on $\Omega$ and attains the two values $2^{4n+3}$ and $2^{4n+3}+3\cdot 2^{2n+2}$ on $\partial\Omega$.
\end{lemma}

In consideration of the $L^\infty$ normalization of the eigenvector $\phi$ the following is immediate from the lemma and Theorem~\ref{landscapeTheorem}.
\begin{corollary}
The constraint~\eqref{eq:highfreqest} is effective in $\Omega$ if $\lambda\geq8\cdot3^{2n+1}$.
\end{corollary}

Observe that for $L_4$ this constraint is $\lambda\geq 157464$.  As expected, this is somewhat higher than the value $\lambda=118038.02$ at which we identified the regime change in the spectrum, because we did not have complete localization at the regime change.  

\begin{remark}
It should be noted that Theorem~\ref{landscapeTheorem} works well in our setting simply because the values in the matrix $L_n$ are very different on boundary edges than on interior edges. This is the same reason that drove our heuristic reasoning in Section~\ref{sec:numericalResults} regarding the decoupling of the interior and boundary energies and their effect on the spectrum of $L$.  The high frequency landscape approach would not be expected to predict whispering gallery modes, or localized modes due to quantum scarring, as these are not due solely to local properties of the Laplacian. 
\end{remark}

\section*{Acknowledgments} The   authors are  grateful to 
Michael Hinz, 
Maria Rosaria Lancia, 
Svitlana Mayboroda, 
Anna Rozanova-Pierrat and 
Tatiana Toro for helpful discussions. 

\bibliographystyle{plain}
\bibliography{snowflake}

\end{document}